\documentclass[12pt]{amsart}

\usepackage{amsmath}
\usepackage{amsfonts}
\usepackage{amssymb}
\usepackage{amsthm}
\usepackage{hyperref}
\usepackage{color}
\usepackage{graphicx}
\usepackage{ytableau}
\usepackage{tikz-cd}

\newtheorem{theorem}{Theorem}[section]
\newtheorem{defn}[theorem]{Definition}
\newtheorem{lem}[theorem]{Lemma}
\newtheorem{cor}[theorem]{Corollary}
\newtheorem{prop}[theorem]{Proposition}
\newtheorem{rem}[theorem]{Remark}
\newtheorem{example}[theorem]{Example}

\DeclareMathOperator{\core}{core}
\DeclareMathOperator{\Core}{\mathcal{C}}
\DeclareMathOperator{\Par}{\mathcal{P}}
\newcommand{\numcore}{c_t}
\newcommand{\coresum}{C_t}
\newcommand{\ds}{\displaystyle}

\begin{document}

\title
[The size of $t$-cores and hook lengths of random cells]
{The size of $t$-cores and hook lengths of random cells in random partitions}

\author{Arvind Ayyer}
\address{Arvind Ayyer, Department of Mathematics, Indian Institute of Science, Bangalore - 560012, India}
\email{arvind@iisc.ac.in}

\author{Shubham Sinha}
\address{Shubham Sinha, Department of Mathematics, University of California San Diego, 9500 Gilman Drive, La Jolla, CA 92093, USA}
\email{shs074@ucsd.edu}

\date{\today}

\begin{abstract}
			Fix $t \geq 2$. We first give an asymptotic formula for certain sums of the number of $t$-cores. 
			We then use this result to compute the distribution of the size of the $t$-core of a uniformly random partition of an integer $n$. 
			We show that this converges weakly to a gamma distribution after dividing by $\sqrt{n}$. As a consequence, we find that the size of the $t$-core is of the order of $\sqrt{n}$ in expectation. We then apply this result to show that the probability that $t$ divides the hook length of a uniformly random cell in a uniformly random partition equals $1/t$ in the limit. Finally, we extend this result to all modulo classes of $t$ using abacus representations for cores and quotients.
\end{abstract}


\maketitle

\section{Introduction and statement of results}

	The irreducible representation of the symmetric group $S_n$ are indexed by partitions of $n$. While studying modular representations of $S_n$, one naturally encounters special partitions called {\em $t$-cores}~\cite{james-kerber-2009}, which are defined for any integer $t \geq 2$. The $t$-core of a partition $\lambda$, denoted $\core_t(\lambda)$, can be obtained from $\lambda$ by a sequence of operations.
	See Section~\ref{sec:core-quo} for the precise definitions.
	A partition $\lambda$ is itself called a $t$-core if $\core_t(\lambda) = \lambda$.
	Let $\numcore(n)$ be the number of $t$-cores of size $n$. 
	We define the number of partitions obtained by taking $t$-cores of all partitions of $n$ by
	\begin{equation}
		\label{def-coresum}
		\coresum(n):= \# \{\core_t(\lambda) \mid \lambda \quad \text{a partition of $n$} \}. 
	\end{equation}
	When $t$ is a prime number, $C_t(n)$ can be defined to be the number of $t$-blocks, that is, the number of connected components of the Brauer graph, in the $t$-modular representation theory of the symmetric group $S_n$~\cite{granville-ono-1996}.
	
	We will show in Proposition~\ref{prop:coresum-formula} that $\coresum(n)$ can be expressed as a sum over $\numcore(n)$ and we 
	are interested in finding the asymptotic behavior of $\coresum(n)$. Although Anderson~\cite[Theorem 2]{anderson-2008} has obtained powerful asymptotic results with precise error estimates for $\numcore(n)$ using the circle method, these will not suffice for our purposes.
	Our first main result is the following.
	
	\begin{theorem}
		\label{thm:sum}
		Fix $t \geq 2$. Then
		\begin{equation}
			\coresum(n)= \frac{(2\pi)^{(t-1)/2} }{t^{(t+2)/2} \, \Gamma(\frac{t+1}{2})} \left( n+\frac{t^2-1}{24} \right)^{{(t-1)/2}} + O(n^{(t-2)/2}),
		\end{equation}
		where $\Gamma$ is the standard gamma function.
	\end{theorem}
	
	We have emphasized the factor $(t^2-1)/24$ in Theorem~\ref{thm:sum} instead of absorbing it into the error term for two reasons. The first is to compare with the result of Anderson~\cite[Theorem 2]{anderson-2008}; see Remark~\ref{rem:anderson} below. The second is that this quantity occurs naturally in our formulas
	for counting lattice points inside an appropriate sphere; see Proposition~\ref{prop:volume}.
	
	\begin{cor}
		\label{cor:numcore-size}
		Fix $t \geq 2$. Then $\numcore(n) = O(n^{(t-2)/2})$.
	\end{cor}
	
	\begin{rem}
		\label{rem:anderson}
		In our notation, Anderson's result~\cite[Theorem~2]{anderson-2008} reads
		\begin{equation}
			\label{anderson}
			\numcore(n) = \frac{(2\pi)^{(t-1)/2} A_t(n)}
			{t^{t/2} \, \Gamma(\frac{t-1}{2})} 
			\left( n+\frac{t^2-1}{24} \right)^{{(t-3)/2}} + O(n^{(t-1)/4})
		\end{equation}
		for $t \geq 6$, where $A_t(n)$ is a complicated number-theoretic double sum, which satisfies 
		$0.05 < A_t(n) < 2.62$ for all $n$ if $t \geq 6$~\cite[Proposition~6]{anderson-2008}. 
		We believe it should also be possible to obtain Theorem~\ref{thm:sum} for $t \geq 6$ as a consequence of Anderson's formula by carefully summing over $A_t(n)$. 
		Our approach is very different and holds for $t \geq 2$.
		Note however that our error term is considerably weaker than the one in \eqref{anderson}. Therefore, there is no hope of using our approach to improve Anderson's result.

		It turns out (see Proposition~\ref{prop:coresum-formula}) that $\numcore(n) = \coresum(n) - \coresum(n-t)$. Therefore $\numcore(n)/t$ can be thought of as a formal derivative of $\coresum(n)$ with respect to $n$.
		If we formally differentiate the expression for $\coresum(n)$ in Theorem~\ref{thm:sum} with respect to $n$ and multiply by $t$, the first term exactly matches with that in \eqref{anderson} except for the factor of $A_t(n)$.
	\end{rem}
	
	Theorem~\ref{thm:sum} and Corollary~\ref{cor:numcore-size} will be proved in Section~\ref{sec:asymp}.
	
	Let $n$ be a positive integer and $t \geq 2$ be a fixed positive integer as before. Let $\lambda$ be a uniformly random partition of $n$.
	Let $Y_n$ be a random variable on $\mathbb{N}_{\geq 0}$ given by
	\begin{equation}
		Y_n \equiv Y_{n,t} = |\core_t(\lambda)|,
	\end{equation}
	where $|\cdot |$ denotes the size of the partition.
	We will be interested in the convergence of $Y_n$. 
	The probability mass function of $Y_n$ is given by
	\begin{equation}
		\label{def-mu}
		\mu_n(k) \equiv \mu_{n,t}(k)=\frac{\#\{\lambda\vdash n:|\core_t(\lambda)|= k \}}{p(n)},
	\end{equation}
	where $p(n)$ is the number of partitions of $n$.
	It can be shown (see Corollary~\ref{cor:number-fixed-cores}) that $\mu_n$
	can be written as 
	\begin{equation}
		\mu_{n}(k) =\frac{\numcore(k) d_t(n-k)}{p(n)},
	\end{equation} 
	where $d_t(m)$ is the number of partitions of $m$ with empty $t$-core.
	Recall that the {\em gamma distribution} with {\em shape parameter} $\alpha > 0$ and {\em rate parameter} $\beta > 0$ is a continuous random variable on $[0,\infty)$ with density given by
	\begin{equation}
		\label{gamma-density}
		\begin{cases}
			\ds \frac{\beta^\alpha}{\Gamma(\alpha)}x^{\alpha-1} \exp(-\beta x), & x\ge 0,\\
			0, & x<0.
		\end{cases}
	\end{equation} 
	The cumulative distribution function (CDF) of the gamma distribution is given by the function $\gamma(\alpha, \beta x)/\Gamma(\alpha)$, where
	\[
	\gamma(s,x) = \int_0^x \; \text{d}t \; t^{s-1} \exp(-t)
	\]
	is the {\em lower incomplete gamma function}.
	
	\begin{theorem}
		\label{thm:conv}
		The random variable $Y_n/\sqrt{n}$ converges weakly and in moments to a gamma-distributed random variable $Y$ with shape parameter $\alpha=(t-1)/2$ and rate parameter $\beta=\pi/\sqrt{6}$.
	\end{theorem}
	
	\begin{center}
		\begin{figure}[htbp!]
			\begin{tabular}{c c}
				\includegraphics[scale=0.29]{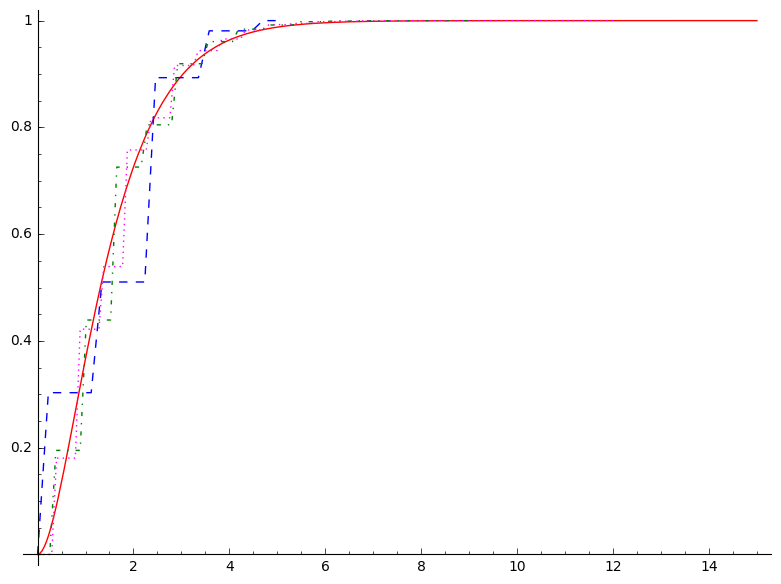} &
				\includegraphics[scale=0.29]{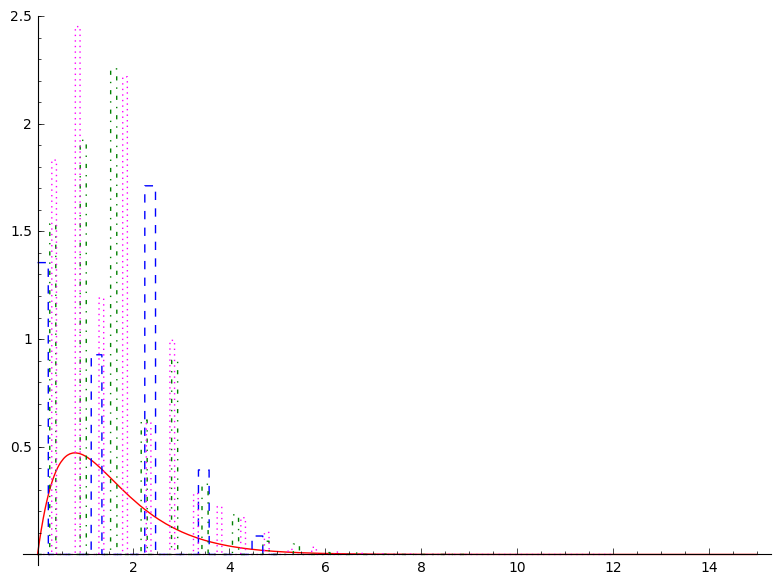} \\
				(a) & (b)
			\end{tabular}
			
			\caption{Comparison of the limiting CDFs and densities for small values of $n$ with $t=5$. A red solid line is used for the limiting distribution, dashed blue for $n = 20$, dash-dotted green for $n = 62$ and dotted magenta for $n=103$.
				In (a) the CDFs, and in (b) the density of $Y$, 
				and the PMFs of these $Y_n/\sqrt{n}$'s as a bar chart with width $1/\sqrt{n}$.}
			\label{fig:dist-core5}
			
		\end{figure}
	\end{center}
	
	See Figure~\ref{fig:dist-core5} for an illustration of Theorem~\ref{thm:conv} for $t=5$. Notice that while the distribution seems to converge pointwise in Figure~\ref{fig:dist-core5}(a), the density in Figure~\ref{fig:dist-core5}(b) does not.
	An immediate consequence of Theorem~\ref{thm:conv} is the following result.
	
	\begin{cor}
		\label{cor:expectation}
		The expectation of the size of $t$-core for a uniformly random partition of size $n$ is asymptotic to $(t-1)\sqrt{6n}/2\pi$.
	\end{cor}
	
	We illustrate Corollary~\ref{cor:expectation} with the example of $t=3$ in Figure~\ref{fig:avg-core3}.
	Theorem~\ref{thm:conv} and Corollary~\ref{cor:expectation} will be proved in Section~\ref{sec:distrib}.
	
	\begin{center}
		\begin{figure}[htbp!]
			\includegraphics[scale=0.5]{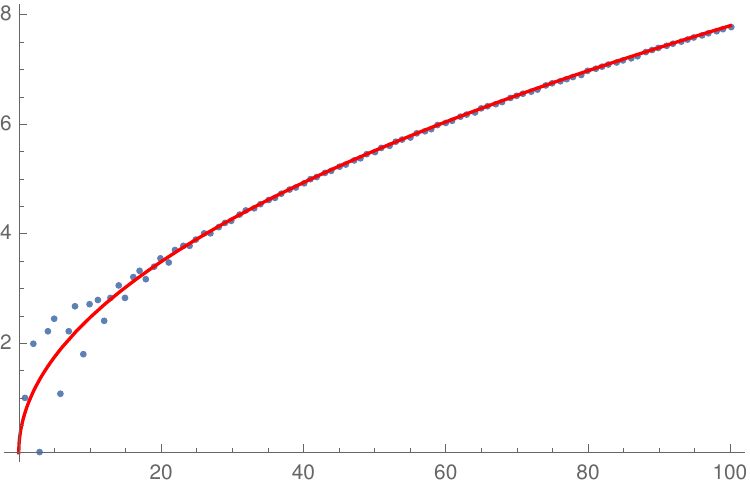}
			\caption{The average size of the $3$-core for partitions of size 1 to 100 in blue circles, along with the result from Corollary~\ref{cor:expectation}, $\sqrt{6x}/\pi$, as a red line.}
			\label{fig:avg-core3}
		\end{figure}
	\end{center}
	
	Let $\lambda$ be a partition of $n$ and $c$ denote a cell in the Young's diagram of $\lambda$. Then we define $h_c$ be the hook length associated to the cell $c$. See Section~\ref{sec:basics} for the precise definitions.
	Our final major result is a statement about the remainder of hook lengths of cells when divided by $t$. 
	
	\begin{theorem}
		\label{thm:random hook length}
		For a uniformly random cell $c$ of a uniformly random partition $\lambda $ of $n$, the probability that the hook length of $c$ in $\lambda$ is congruent to $i$ modulo $t$ is asymptotic to $1/t$ for any $i\in \{0,1,\dots, t-1 \}$.
	\end{theorem}
	
	When $i=0$, this will follow as a corollary of Theorem~\ref{thm:conv}. 
	Theorem~\ref{thm:random hook length} will be proved in Section~\ref{sec:hooks}.
	
	We begin with preliminaries in Section~\ref{sec:prelim}. We first recall the basics of partitions in Section~\ref{sec:basics}. The abacus representation for partitions is given in Section~\ref{sec:abacus} and basic definitions and results for cores and quotients of partitions are recalled in Section~\ref{sec:core-quo}. 
	
	{\bf Note added in proof}: After this work was completed, we were made aware of gem of a paper of Lulov and Pittel~\cite{lulov-pittel-1999} which seems to have been completely missed by almost everyone. Several results in the literature have rederived or improved the results here. For example, 
	Theorem 2 of Anderson~\cite{anderson-2008} is entirely subsumed by Lemma 5(a) of this paper. Further, Section 5 and especially Theorem 3 of this paper deals with the shape of large random cores and proves a more general result than Proposition 4 of \cite{lam-2015} (which rests on Conjecture 2 therein, now proved in \cite{ayyer-linusson-2016}). Neither author seems aware of this paper.
	
	The Lulov--Pittel paper includes many of our results, but with different proofs. Our Theorem~\ref{thm:sum} is similar in spirit to their Lemma 5(b). In fact, their error estimates are better and we believe our results can follow from theirs.
	However, our proofs are more elementary and geometric in nature, whereas theirs is more technical and number-theoretic. 
	Our Theorem~\ref{thm:conv} is identical to their Theorem 2. However, our Theorem~\ref{thm:random hook length} and all the results in Section~\ref{sec:hooks} are completely new.
	
	See also a recent nice preprint of Rostam~\cite{rostam-2021} for analogous results for the Plancherel measure on partitions.

	\section{Preliminaries}
	\label{sec:prelim}
	
	\subsection{Basics}
	\label{sec:basics}
	
	Recall that an {\em (integer) partition} $\lambda$ of a nonnegative integer $n$ is a nonincreasing tuple of positive integers which sum up to $n$.
	If $\lambda$ is a partition of $n$, we write $\lambda \vdash n$ and say that the {\em size} of $\lambda$, denoted $|\lambda|$, is $n$. Let $p(n)$ denote the number of partitions of $n$. It is easy to see that the ordinary generating function of $p(n)$ is given by
	\begin{equation}
		\label{partn-gf}
		\sum_{n \geq 0} p(n) x^n = \prod_{j \geq 1} \frac{1}{1-x^j}.
	\end{equation}
	Partitions also arise in a natural way in the representation theory of the symmetric groups. The irreducible representations of $S_n$ are indexed by partitions of $n$~\cite{james-kerber-2009}. We write $\Par$ for the set of all partitions.
	It is well-known that the asymptotics of $p(n)$ is given by~\cite[Equation~(5.1.2)]{andrews-1976}
	\begin{equation}
		\label{asy_p}
		p(n)\sim\frac{\exp \left(\pi \sqrt{\frac{2n}{3}} \right)}{4n\sqrt{3}}.
	\end{equation}
	
	A partition $\lambda = (\lambda_1,\dots,\lambda_k)$ has a graphical representation as a {\em Young diagram} or {\em shape}, in which we place $\lambda_i$ left-justified boxes in the $i$'th row, for $1 \leq i \leq k$. For example, 
	$(5,4,4,2,1)$ is a partition of $16$ whose Young diagram (in English notation) is
	\begin{equation}
		\label{eg-ydiag}
		\ydiagram{5,4,4,2,1}.
	\end{equation}
	The {\em conjugate partition} of $\lambda$, denoted $\lambda'$ is a partition of the same size whose Young diagram is obtained by taking the transpose of the Young diagram of $\lambda$. In the above example, $(5,4,4,2,1)' = (5,4,3,3,1)$.
	A {\em standard Young tableau} (SYT) is a filling of a Young diagram of $\lambda \vdash n$ with entries in $\{1, \dots, n\}$ such that the entries strictly increase as we read from left to right and from top to bottom. 
	Let $c \equiv (i,j)$ be a cell in the Young diagram of $\lambda$.
	The {\em hook} of $c$ is a subset of the cells in the Young diagram containing the cells to the right of $c$ in the same row, those below $c$ in the same column, and $c$ itself. 
	The {\em hook length} is the number of the cells in the hook of $c$ and is given by $h^\lambda_c \equiv h_c = \lambda_i - i + \lambda'_j - j +1$. 
	The part $\lambda_i - i$ is known as the {\em arm length} and $\lambda'_j - j$ is known as the {\em leg length}.
	The rightmost (resp. bottommost) cell in the hook of $c$ is called the {\em arm node} (resp. {\em leg node}) of $c$. The {\em max-hook} is the hook of the cell $(1,1)$. The {\em content} of $c$ is $j-i$.
	The hook numbers and contents of cells in the previous example of $(5,4,4,2,1)$ are
	\begin{equation}
		\label{eg-hooks-contents}
		\ytableaushort{97541,7532,6421,31,1}
		\quad
		\raisebox{-1.1cm}{\text{and}}
		\quad
		\ytableaushort{01234,{-1}012,{-2}{-1}01,{-3}{-2},{-4}}
		\quad
		\raisebox{-1.1cm}{\text{respectively.}}
	\end{equation}
	
	\subsection{Abacus representation}
	\label{sec:abacus}
	
	\begin{defn}
		\label{def:abacus}
		An {\em abacus} or {\em (1-runner)} is a function $w : \mathbb{Z} \to \{0,1\}$ such that there exist $m,n \in \mathbb{Z}$ such that $w_i =1$ (resp. $w_i =0$) for all $i \leq m$ (resp. $i \geq n$).
	\end{defn}
	
	Starting from an abacus $w$, consider the up-right path formed by replacing $1$'s by vertical steps and $0$'s by horizontal steps. This path will form the outer boundary of a partition. For example, the abacus
	\begin{equation}
		\label{eg-abacus}
		(\dots, 1,1,1,0,1,0,1,0,\underline{0},1,1,0,1,0,0,0,\dots)
	\end{equation}
	maps to the partition in \eqref{eg-ydiag}, where the bit at position $0$ is marked by underlining it. However, note that this is not a bijective correspondence because any shift of the abacus will lead to the same partition. An abacus $w$ is called {\em justified at position $p$} if $w_i = 1$ (resp. $w_i = 0$) for $i < p$ (resp. $i \geq p$). An abacus is called {\em justified} if it is justified at position $p$ for some $p$. 
	Any abacus can be transformed to a justified one by moving the $1$'s to the left past the $0$'s starting from the leftmost movable $1$. 
	An abacus is called {\em balanced} if, after this transformation, it is justified at position $0$. The abacus in \eqref{eg-abacus} is balanced.
	Note that balanced abaci are in bijection with partitions.
	
	We now summarize properties of abaci that will be relevant to us. Readers interested in the details can look at~\cite{james-kerber-2009,loehr-2011,olsson}. 
	
	\begin{prop}
		\label{prop:core-facts}
		Let $\lambda$ be a partition with corresponding balanced abacus $w$ and $c$ be a cell in the Young diagram of $\lambda$.
		Then the following properties hold.
		\begin{enumerate}
			
			\item Cells in the Young diagram of $\lambda$ are in bijection with pairs $(i,j)$ such that $i < j$, $w_i = 0$ and $w_j = 1$.
			
			\item The hook length of the cell $c$ corresponding to the pair $(i,j)$ in $w$ is $j-i$. The arm length (resp. leg length) of $c$ is the number of $0$'s (resp. $1$'s) between $i$ and $j$ in $w$ (excluding these).
			
			\item Suppose $c$ has hook length $t$. Then $c$ corresponds to a pair $(i,i+t)$ in $w$. Then, removing a $t$-rim hook for the cell $c$ (see the beginning of Section~\ref{sec:core-quo}) from the Young diagram of $\lambda$ amounts to exchanging the $i$'th and $(i+t)$'th entries in $w$.
			
		\end{enumerate}
		
	\end{prop}
	
	\subsection{Cores and quotients}
	\label{sec:core-quo}
	
	To describe cores and quotients, we will need some notation. Associated to a partition $\lambda$ and a cell $c$ in its Young diagram, the set of cells joining the two corners in the hook of $c$ along the boundary is known as the {\em rim hook} or {\em ribbon} of $c$. In the running example of \eqref{eg-ydiag} below,
	\[
	\ytableaushort{{}{}{}{}{},c{}{}\cdot,
		{}\cdot\cdot\cdot,\cdot\cdot,\cdot},
	\]
	the rim hook of the cell marked $c$ are the cells marked by $\cdot$'s. Note that the number of cells in the rim hook of $c$ is given by $h_c = 7$. It is easy to see that the cardinality of the rim hook of any cell is also the hook number of that cell.
	
	The {\em $t$-core of a partition} $\lambda$, denoted $\core_t(\lambda)$, is the partition obtained by removing as many rim hooks of size $t$ from $\lambda$ as possible. A nontrivial fact (see \cite{olsson}, for example) is that the resulting partition is unique. Although this is not obvious in the Young diagram picture, this is easy to see in the abacus representation. By Proposition~\ref{prop:core-facts}(3), one can exchange $(0,1)$ pairs at positions $(i,i+t)$ for each class modulo $t$ independently.
	Going back to the running example in \eqref{eg-ydiag}, $\core_3((5,4,4,2,1)) = (2,1,1)$. And from the corresponding abacus in \eqref{eg-abacus}, the $3$-core is the corresponding balanced abacus,
	\[
	(\dots, 1,1,1,1,1,0,1,1,\underline{0},1,0,0,0,0,0,0,\dots).
	\]
	
	A partition $\lambda$ is called a {\em $t$-core} if $\core_t(\lambda) = \lambda$, or equivalently, if none of the hook numbers in the Young diagram of $\lambda$ is equal to $t$.\footnote{It is slightly more nontrivial to prove that $\lambda$ is a $t$-core if none of the hook numbers are {\em divisible} by $t$, but this can be seen using abaci; see~\cite[Section I.3]{olsson}.}
	From Proposition~\ref{prop:core-facts}(3), it follows that the partition $\lambda$ is a $t$-core if and only if there is no $i \in \mathbb{Z}$ such that $w_i = 0$ and $w_{i+t} = 1$.
	The partition $(5,4,4,2,1)$ is not a $t$-core for any $2 \leq t \leq 7$, but it is an $8$-core. Let $\Core_t$ be set of all $t$-cores.
	Let $\numcore(n)$ be the number of $t$-cores of size $n$.
	Then it is known~\cite{klyachko-1982} that
	\begin{equation}
		\label{coret-gf}
		\sum_{n=0}^\infty \numcore(n)x^n = \prod_{k=1}^\infty 
		\frac{(1-x^{tk})^t}{1-x^{k}}.
	\end{equation}
	Note that the only $2$-cores are staircase partitions $(m,m-1,\dots,1)$, and therefore
	\begin{equation}
		c_2(n) = 
		\begin{cases}
			1 & n = k(k-1)/2 \text{ for some $k$}, \\
			0 & \text{otherwise}.
		\end{cases}
	\end{equation}

	\begin{defn}
		\label{def:t-divisible}
		A partition is said to be {\em $t$-divisible} if it has empty $t$-core.
	\end{defn}
	
	Let $\mathcal{D}_t(n)$ be the set of $t$-divisible partitions of $n$ and 
	$d_t(n)$ be its cardinality.
	By a result of Garvan, Kim and Stanton~\cite{garvan-kim-stanton-1990}, the generating function of $d_t(n)$ is given by
	\begin{equation}
		\label{emptycore-gf}
		\sum_{n=0}^\infty d_t(n)x^n = \prod_{k=1}^\infty \frac{1}{(1-x^{tk})^t}.
	\end{equation}
	
	The {\em $t$-quotient} of $\lambda$ is a $t$-tuple of partitions $(\lambda^0, \dots, \lambda^{t-1})$ described as follows. For each $k$ between $0$ and $t-1$, $\lambda^k$ consists of the cells of $\lambda$ whose hook lengths are divisible by $t$ and whose arm nodes have content congruent to $k$ modulo $t$.
	It is a nontrivial fact that this subdiagram is itself a Young diagram.
	
	Quotients are also naturally understood in terms of abaci.
	Given a partition $\lambda$ with abacus $w$, construct $t$ abaci by letting $\lambda^{i} = (w_{n t + i})_{n \in \mathbb{Z}}$ for $0 \leq i \leq t-1$. The {\em $t$-runner abacus} of $\lambda$ is then $(\lambda^0,\dots,\lambda^{t-1})$ of $1$-runner abaci, where the zeroth position in $\lambda^0$ is underlined. Notice that we have used the same notation for the $i$'th entry of the $t$-quotient and the $t$-runner. From the context, the usage should be clear.
	This is because the $i$'th runner in the $t$-abacus, when interpreted as a partition is exactly the $i$'th entry of the $t$-quotient.
	
	Going back to our running example in \eqref{eg-ydiag}, the $3$-quotient of $(5,4,4,\allowbreak 2,1)$ is seen to be $(\emptyset,(1,1,1),(1))$ by looking at its hooks and contents in \eqref{eg-hooks-contents}.
	The corresponding $3$-runner abacus is 
	\[
	\begin{matrix}
		\lambda^0 = & (\dots, & 1, & 1, & 0, & \underline{0}, & 0, & 0, & \dots), \\
		\lambda^1 = & (\dots, & 1, & 0, & 1, & 1, & 1, & 0, & \dots), \\
		\lambda^2 = & (\dots, & 1, & 1, & 0, & 1, & 0, & 0, & \dots). \\
	\end{matrix}
	\]
	Notice that not all $i$'th runners are balanced. 
	It turns out that $t$-cores and $t$-divisible partitions have a natural interpretation in terms of $t$-runner abaci. 
	
	\begin{prop}
		\label{prop:t-runner}
		Let $\lambda$ be a partition and $(\lambda^0,\dots,\lambda^{t-1})$ be its $t$-runner abacus. Then
		\begin{enumerate}
			\item $\lambda$ is a $t$-core if and only if $\lambda^i$ is justified for $0 \leq i \leq t-1$ and therefore its positions of justification $(p_0,\dots,p_{t-1})$ sum to zero, and 
			
			\item $\lambda$ is $t$-divisible if and only if $\lambda^i$ is balanced for $0 \leq i \leq t-1$.
		\end{enumerate}
	\end{prop}

	A fundamental result on cores and quotients is the {\em partition division theorem}~\cite[Theorem~11.22]{loehr-2011}, which is an analogue of the division algorithm for integers. We restate it in slightly different terminology more suited to our purposes.
	
	\begin{theorem}
		\label{thm:partn-divn}
		Let $t \geq 2$ be an integer. Then there is a natural bijection 
		\[
		\Delta_t: \Par \to \Core_t \times \mathcal{D}_t
		\]
		defined by $\Delta_t(\lambda)=(\rho, \nu )$, where $\rho$ is the $t$-core of $\lambda$ and $\nu$ is a $t$-divisible partition whose $t$-quotient is $(\nu^0,\dots,\nu^{t-1})$. Moreover,
		\[
		|\lambda| = |\rho| + |\nu| = |\rho| + t \sum_{k=0}^{t-1} |\nu^k|.
		\]
	\end{theorem}
	
	\begin{rem}\label{rem:t-runner_shift}
		The bijection $\Delta_t$ is best understood in terms of $t$-runner abacus. Given a partition $\lambda$, let $(p_0,\dots p_{t-1})$ be the positions of justification of $\core_t(\lambda)=\rho$ described in Proposition \ref{prop:t-runner}. Let $\delta$ denote the left shift on $1$-runner abaci, i.e., 
		$\delta(\dots,w_{-1}, w_0, w_1, \dots) = (\dots, w_0, w_1, w_2, \dots)$. Then the $t$-divisible partition $\nu$, obtained by $\Delta_t$, is the partition corresponding to the $t$-runner \[(\nu^0,\dots,\nu^{t-1})=(\delta^{p_0}\lambda^0,\dots , \delta^{p_{t-1}}\lambda^{t-1}),\] where $(\lambda^0,\dots,\lambda^{t-1})$ is the $t$-runner abacus of $\lambda$.
	\end{rem}
	
	The following corollary is then immediate.
	
	\begin{cor}
		\label{cor:number-fixed-cores}
		Let $i,n$ be positive integers such that $i < n$. Then 
		\[
		\#\{\lambda\vdash n \mid |\core_t(\lambda)|=i \} = d_t(n-i)\numcore(i).
		\]
	\end{cor}
	
	Since $d_t(n)=0$ for $t\nmid n$, the asymptotics of $d_t(n)$ for $t|n$ is given by a special case of \cite[Theorem~6.2]{andrews-1976} as 
	\begin{equation}
		\label{asy_d}
		d_t(n)\sim \frac{t^{(t+2)/2} \exp \left(\pi \sqrt{\frac{2n}{3}} \right)}{2^{(3t+5)/4} 3^{(t+1)/4}n^{(t+3)/4 } }.
	\end{equation}

	\section{Asymptotics of the number of $t$-cores}
	\label{sec:asymp}
	The asymptotics of the number of $t$-cores, $\numcore(n)$, was obtained using the circle method by Anderson~\cite[Theorem~2]{anderson-2008} for $t \geq 6$. We demonstrate a completely new method to obtain these asymptotics for $t \geq 2$.
	
	The following theorem describe the quantity $\numcore(n)$ as the number of integer solutions of a particular quadratic equation. This is a reformulation of \cite[Bijection 2]{garvan-kim-stanton-1990}, but we will give a new proof using the abacus representation.
	Let $H_t$ denote the hyperplane in $\mathbb{R}^t$ given by
	\[
	H_t = \{ (x_0,\dots,x_{t-1}) \mid x_0+x_1+\dots x_{t-1}=0 \}.
	\]
	Notice that our indices are labelled $0$ through $t-1$. Define 
	$F_t : \mathbb{R}^t \to \mathbb{R}$ by
	\begin{equation}
		\label{sphere}
		F_t(p) = \frac{t}{2} \sum_{i=0}^{t-1}p_i^2
		+ \sum_{i=0}^{t-1}ip_i.
	\end{equation} 
	
	\begin{rem}
		\label{rem:sphere}
		Since, for $p\in H_t $, $\sum_{i=0}^{t-1}p_i=0$, we subtract $(t-1) \allowbreak \times\left( \sum_{i=0}^{t-1}p_i \right)/2$ from the right hand side of~\eqref{sphere} to get 
		\begin{align}
			\label{sphere-2}
			F_t(p)=\sum_{i=0}^{t-1} \frac{t}{2}\bigg( p_i- \frac{t-1-2i}{2t} \bigg)^2-\frac{t^2-1}{24}.
		\end{align}
		Note that $F_t(p)=0$ is the formula for a $t$-dimensional sphere centered at $((t-1-2i)/(2t))_{0\le i\le t-1} $.
		The minimizer of $F_t(p)$ on integral points in $H_t$ is the closest point to the center of the sphere which is $p=(0,\dots,0)$, for which $F_t(p)=0$. Therefore if $p \in H_t\cap \mathbb{Z}^t$, $F_t(p)$ is a nonnegative integer.
	\end{rem}
	
	\begin{theorem}
		\label{thm:eq-sols}
		The number of $t$-cores of $n$, $\numcore(n)$, is equal to the number of integer solutions of $F_t(p) = n$ on the hyperplane $H_t$.
	\end{theorem}
	
	\begin{proof}
		Let $p = (p_0, \dots, p_{t-1}) \in H_t \cap \mathbb{Z}^t$. By considering $p_i$ as the position of justification of the $i$'th runner in a $t$-runner abacus for a $t$-core, we see that such $p$'s
		are naturally in bijection with $t$-cores by Proposition~\ref{prop:t-runner}(1). Therefore, it suffices to show that the positions of justification of $t$-cores of size $n$ satisfy $F_t(p) = n$.
		Suppose $\lambda$ is a $t$-core and $p  = (p_0,\dots, p_{t-1})$ records the positions of justification of the $i$'th runner. We then claim that $F_t(p) = |\lambda|$. 
		
		We will prove that $p \in H_t$ and $F_t(p) = n$ by induction on $n$. For the base case of $n=0$, it can be verified that the empty partition is the unique $t$-core, for which $p=(0,\dots,0)$. 
		
		Suppose $n \geq 1$ and $\lambda \vdash n$.
		Observe that removing the max-hook from $\lambda$ yields a $t$-core, $\mu$ say, since it does not change the hook lengths of the remaining cells. Removing a hook corresponds to swapping a $(0,1)$ pair in its $t$-runner abacus. Thus $\mu$ has $p' = (p_0,..,p_i-1,..,p_j+1,..,p_{t-1})$ as the positions of justifications of its $t$-runner (the relative positions of $i$ and $j$ are not important). The required statements follow by the induction hypothesis, by noting
		that the size of the max-hook is $F_t(p) - F_t(p') = t(p_i-p_j-1)+(i-j)$. 
	\end{proof}

	Equating $F_t(p)=n$ using \eqref{sphere-2} yields the equation of a $t$-dimensional sphere centered at a point in the hyperplane $H_t$ given by
	\begin{equation}
		\label{c_sphere}
		\sum_{i=0}^{t-1} \bigg( p_i- \frac{t-1-2i}{2t} \bigg)^2 = \frac{2}{t}\bigg(n + \frac{t^2-1}{24}\bigg).
	\end{equation}
	We will now compute the volume of the $(t-1)$-dimensional ball $B_{t-1}(n)$ whose boundary is given by \eqref{c_sphere} and contained in the hyperplane $H_t$.
	The radius of $B_{t-1}(n)$ is unchanged since the center 
	$(t-1, t-3, \dots, -(t-3), -(t-1))/(2t)$
	lies in $H_t$. The following proposition follows from the standard formula for the volume of a $d$-dimensional sphere of radius $r$.
	
	\begin{prop}
		\label{prop:volume}
		The volume $V_t(n)$ of $B_{t-1}(n)$ under the induced measure on the hyperplane $H_t$, is 
		\[
		V_t(n)=\frac{1}{\Gamma(\frac{t+1}{2})} \bigg(\frac{2\pi}{t} \left( n+\frac{t^2-1}{24} \right) \bigg)^{(t-1)/2}.
		\]
	\end{prop}
	
	Define the lattice $\Lambda_t = H_t\cap \mathbb{Z}^t$. It is clear that $\Lambda_t$ is a full lattice of the codimension one subspace $H_t$ with $\mathbb{Z}$-basis $\{e_0-e_1,e_0-e_2,\dots, e_0-e_{t-1} \}$. We choose the parallelepiped formed by the $\mathbb{Z}$-basis of a lattice to be a fundamental domain of the lattice.
	
	\begin{prop}
		\label{prop:fun_domain}
		The volume $V(\Lambda_t)$ of the fundamental domain of $\Lambda_t$ in $H_t$ is $\sqrt{t}$.
	\end{prop}
	
	\begin{proof}
		It is a standard fact that the volume $V(\Lambda)$ of a parallelepiped $\Lambda$ spanned by $m$ linearly independent vectors $v_1, \dots, v_m$ in $\mathbb{R}^n$ can be computed by taking the square root of the determinant of the corresponding $m \times m$ Gram matrix $(\langle v_i, v_j \rangle)$. To see this, let $A$ be the $m \times n$ matrix whose rows are the vectors. Then it is clear that 
		\begin{equation}
			V(\Lambda_t)^2= \det A \, A^t = \det (\langle\alpha_i,\alpha_j\rangle)_{i,j=1}^m.
		\end{equation}
		In our case, $A=(\alpha_i)$ is a $(t-1) \times t$ matrix where $\alpha_i=e_0-e_i$. Now note that $\langle \alpha_i,\alpha_j\rangle $ equals 1 when $i\ne j$ and 2 when $i=j$. Thus $A \, A^t = I + J$, where $J$ is the all-ones matrix. The determinant of $A \, A^t$ is easily computed to be $t$.
	\end{proof}
	
	The following is a standard result in number theory.
	
	\begin{prop}[{\cite{vinogradov-1979}}]
		\label{prop:volume-ellipsoid}
		Let $\Lambda$ be a full lattice in $\mathbb{R}^{d}$ with fundamental domain of volume $V(\Lambda)$. Let $B_d(r,p)$ be a ball of radius $r$ centered at some fixed point $p$ with volume $V_d(r)$ . Then
		\begin{equation}
			\big|\Lambda \cap B_d(r,p)\big|= \frac{V_d(r)}{V(\Lambda)} +
			O(r^{d-1}).
		\end{equation}
	\end{prop}
	
	Recall that the number of partitions obtained by taking $t$-cores of all partitions of $n$ is denoted by $\coresum(n)$ and is defined in\eqref{def-coresum}.
	
	\begin{prop}
		\label{prop:coresum-formula}
		\[
		\coresum(n) = \sum_{i=0}^{\lfloor \frac{n}{t}\rfloor} \numcore(n - i \, t).
		\]
	\end{prop}
	
	\begin{proof}
		Observe that for any $\lambda \vdash n$, $|\core_t(\lambda)|\equiv |\lambda|\mod{t}$. Hence, from Theorem~\ref{thm:partn-divn}, we obtain 
		\begin{equation*}
			\coresum (n)=\# \{\core_t(\lambda) : \lambda \vdash n \} \leq \sum_{i=0}^{\lfloor\frac{n}{t}\rfloor }\numcore(n-it).
		\end{equation*} 
		The reverse inequality holds because for any $t$-core $\mu$ of size $n-jt$, where $0 \leq j \leq \lfloor n/t \rfloor$, there exists a $\lambda\vdash n$ such that $\core_t(\lambda)=\mu$. One simple way to construct such a $\lambda$ is to add $jt$ to the first part of $\mu$. 
	\end{proof}
	
	For the next result, we will work over the module $(\mathbb{Z}/t\mathbb{Z})^{t}$. To avoid confusion, we will denote points in $(\mathbb{Z}/t\mathbb{Z})^{t}$ with a tilde, e.g. $\widetilde{P},\widetilde{Q}$.
	
	\begin{lem}
		\label{lem:size-Zmodt}
		For any $n \in \mathbb{N}$, the cardinality of the set 
		$ \big\{\widetilde Q\in \widetilde{H}_t \mid F_t(\widetilde Q)\equiv n\mod{t} \big\}$ is $t^{t-2}$,
		where $F_t$ is defined in \eqref{sphere} and $\widetilde{H}_t$ is the hyperplane of $(\mathbb{Z}/t\mathbb{Z})^{t}$ defined by $q_0+q_1+\dots +q_{t-1}=0$.
	\end{lem}
	
	\begin{proof}
		Observe that $F_t(x_0,x_1,\dots x_{t-1})\equiv x_1+2x_2+\dots +(t-1)x_{t-1}$ modulo $t$. It is enough to note that for any tuple $(q_2,q_3,\dots q_{t-1})\in (\mathbb{Z}/t\mathbb{Z})^{t-2}$, there exist a unique solution for $(x_0,x_1)$ satisfying the equations 
		\begin{align*}
			x_0+x_1 &= -(q_2+q_3+\dots+q_{t-1}), \\
			x_1 &= -(2q_2+3q_3+\dots +(t-1)q_{t-1}),
		\end{align*}
		in $\mathbb{Z}/t\mathbb{Z}$. 
	\end{proof}
	
	We are now in a position to prove the main result of this section. Recall that $\Lambda_t = H_t\cap \mathbb{Z}^t$.
	
	\begin{proof}[Proof of Theorem~\ref{thm:sum}]
		From Theorem~\ref{thm:eq-sols} and Proposition~\ref{prop:coresum-formula}, we 
		have that 
		\begin{equation}
			\coresum(n)=\# \big\{Q \in \Lambda_t \cap B_{t-1}(n) \mid F_t(Q)\equiv n \mod t \big\}.
		\end{equation}
		Remark \ref{rem:sphere} ensures that $F_t(Q)$ is never negative.
		Splitting the set of points in the set on the right hand side according to their remainder modulo $t$, we arrive at 
		\begin{equation}
			\coresum(n)=
			\sum_{\substack{\widetilde{P} \in \Lambda_{t}/t\Lambda_{t} \\ 
					F(\widetilde{P}) \equiv n \mod{t} }}
			\# \left( (\widetilde{P}+t\Lambda_{t}) \cap B_{t-1}(n) \right),
		\end{equation}
		where $\widetilde P+t\Lambda_{t}$ is the lattice $t \Lambda_t$ shifted by $\widetilde{P}$. 
		Clearly, the fundamental domain has volume  $V(\widetilde P+t\Lambda_{t}) =t^{t-1}V(\Lambda_{t})$. 
		Using Proposition~\ref{prop:volume-ellipsoid}, 
		\begin{equation}
			\coresum(n) = \sum_{\substack{\widetilde{P}\in \Lambda_{t}/t\Lambda_{t} \\ F(\widetilde{P})\equiv n \mod{t} }}
			\Bigg(\frac{V_t(n)}{t^{t-1}V(\Lambda_{t})}+ O \big(n^{(t-2)/2}\big) \Bigg),
		\end{equation}
		where $V_t(n)$ is computed in Proposition~\ref{prop:volume}. 
		Now, use Lemma~\ref{lem:size-Zmodt} to see that each modulo class gives the same result and thus Proposition~\ref{prop:fun_domain} yields 
		\begin{align*}
			\coresum(n) &=  \Bigg(\frac{V_t(n)}{t^{t-1}V(\Lambda_{t})}+ O\big(n^{(t-2)/2}\big) \Bigg)t^{t-2}\\
			&= \frac{V_t(n)}{t^{3/2}}+ O(n^{(t-2)/2}),
		\end{align*}
		leading to the desired result.
	\end{proof}
	
	\begin{proof}[Proof of Corollary~\ref{cor:numcore-size}]
		First, note that $\numcore(n)=\coresum(n)-\coresum(n-t)$ by Proposition~\ref{prop:coresum-formula}. 
		From the proof of Theorem~\ref{thm:sum} and the mean value theorem, it follows that
		\[
		\numcore(n) = \frac{V_t(n)-V_t(n-t)}{t^{3/2}}+O(n^{(t-2)/2}) 
		= \frac{V_t'(x)}{t^{1/2}} + O(n^{(t-2)/2})
		\]
		for some $x\in [n-t,n]$, where $V_t(x)$ is given in Proposition~\ref{prop:volume}. 
		We remark that the first term here is exactly the same as the first term in Anderson's formula~\cite[Theorem 2]{anderson-2008}.
		Our result follows because $V_t'(n)$ is $O(n^{(t-3)/2})$ and increasing as a function of $n$ when $n$ is large enough.
	\end{proof}

	\section{Distribution of core sizes}
	\label{sec:distrib}
	
	In this section we will give the proof of Theorem~\ref{thm:conv} by calculating the moments of $X_n$. 
	As usual, we fix $t \geq 2$.
	Let $\epsilon(n) \in \{0,1,\dots,t-1\}$ be the remainder when $n$ is divided by $t$. 
	For convenience, define
	\begin{equation}
		\label{def-ln}
		\ell_n(y) = t\lfloor y\sqrt{n}\rfloor+\epsilon(n), \quad y\ge 0.
	\end{equation}
	Notice that for $y$ large enough, 
	\begin{equation}\label{eq:inequality_l_n(y)}
		\frac{1}{2}<\frac{yt\sqrt{n}}{\ell_n(y)}< 2.
	\end{equation}
	We will make crucial use of the functions
	\begin{equation}
		\label{def-psi-n}
		\psi_n(y) = \frac{\numcore(\ell_n(y))}{n^{(t-3)/4}},
	\end{equation}
	and
	\begin{equation}
		\label{def-gnk}
		g_{n,k}(y) = 
		\begin{cases}
			\ds \ell_n(y)^{k} n^{(t-1)/4-k/2} \frac{d_t(n - \ell_n(y))}{p(n)} & y > 0, \\
			0 & y \le 0.
		\end{cases}
	\end{equation}
	Recall the PMF $\mu_n(x)$ defined in \eqref{def-mu}.
	The functions $g_{n,k}$ and $\psi_n$ are defined so that $\psi_n(y) g_{n,0}(y)=\sqrt{n}\mu_n(\ell_n(y))$.
	
	\begin{lem}
		\label{lem:gn-bound}
		There exists a polynomial $Q_k$ and a constant $\theta$ independent of $n$, such that, for all $y \in \mathbb{R}$, 
		\begin{equation}
			|g_{n,k}(y)| < Q_k(y) e^{-\theta y}.
		\end{equation}
	\end{lem}
	\begin{proof}
		Using the asymptotic results in \eqref{asy_p}, \eqref{asy_d} and the fact that $p$ and $d_t$  are functions on natural numbers, we 
		obtain that for suitably chosen positive constants $c_1$ and $c_2$,
		\begin{align*}
			p(n) > \frac{c_1  \exp \left(\pi\sqrt{\frac{2n}{3}} \right)}{n}, \quad
			d_t(n) < \frac{c_2  \exp \left(\pi\sqrt{\frac{2n}{3}} \right)}
			{(n+1)^{(t+3)/4}}
		\end{align*}
		for all $n\in \mathbb{N}$. Now, there are three cases depending on the value of $y$.
		
		\noindent
		\textbf{Case 1}: $\ell_n(y) \leq n/2$: \\
		Observe that for $y>1$, we have  
		\[
		\sqrt{n}-\sqrt{n - \ell_n(y)} > \frac{\ell_n(y)}{2\sqrt{n}}>\frac{y}{2}
		\]
		for all $n \ge 1$.
		Therefore, we obtain
		\begin{multline*}
			\frac{n^{(t-1)/4} d_t(n - \ell_n(y)) }{p(n)} \\
			< n^{(t-1)/4} \frac{c_2 \exp \left(\pi\sqrt{\frac{2}{3}(n - \ell_n(y))} \right) }
			{ (n+1 - \ell_n(y))^{(t+3)/4} } \frac{n} {c_1 \exp \left(\pi\sqrt{\frac{2n}{3}} \right)}\\
			<  \frac{c_2}{c_1} 2^{(t+3)/4} \exp \left(-\frac{\pi y}{\sqrt{6}} \right).
		\end{multline*}
		Hence 
		\[
		g_{n,k}(y)< \frac{c_2}{c_1} t^k y^k 2^{(t+3)/4} \exp \left(-\frac{\pi y}{\sqrt{6}} \right).
		\]
		
		\noindent
		\textbf{Case 2}: $n/2 < \ell_n(y) \le n$:\\
		Here, we have $\sqrt{n}/(4t) < y < 2\sqrt{n}/t$ by \eqref{eq:inequality_l_n(y)}, so that $y^2t^2/4 < n < 16y^2t^2$. 
		Using this fact, we obtain 
		\begin{multline*}
			\frac{n^{(t-1)/4} d_t(n - \ell_n(y)) }{p(n)} \\
			<  n^{(t-1)/4} \frac{c_2 \exp \left(\pi\sqrt{\frac{2}{3}(n - \ell_n(y))} \right) }{ (n+1 - \ell_n(y))^{(t+3)/4} } \frac{n} {c_1 \exp \left(\pi\sqrt{\frac{2n}{3}} \right)}\\
			< \frac{c_2}{c_1} n^{(t+3)/4}\exp \left(\pi \sqrt{\frac{2}{3}} \left(\sqrt{\frac{n}{2}}-\sqrt{n} \right) \right) \\
			< \frac{c_2}{c_1} (16y^2t^2)^{(t+3)/4} \exp \left(\pi {\frac{yt}{\sqrt{6}}} \left(\frac{1}{\sqrt{2}}-1 \right) \right).
		\end{multline*}
		Hence we get
		\[ 
		g_{n,k}(y)< \frac{c_2}{c_1} t^ky^k(16y^2t^2)^{(t+3)/4} \exp \left(\pi {\frac{yt}{\sqrt{6}}} \left(\frac{1}{\sqrt{2}}-1 \right)  \right).
		\]
		
		\noindent
		\textbf{Case 3}: $\ell_n(y) > n$:\\
		This is trivial because $g_{n,k}(y)=0$ here.
		
		\medskip
		\noindent
		Combining these three cases proves the result.
	\end{proof}

	We now define
	\begin{equation}
		\label{def:g_k}
		g_k(y)= \begin{cases}
			\kappa y^{k}e^{-\beta t y} & y\ge 0, \\
			0 & y < 0,
		\end{cases}
	\end{equation}
	where $\kappa=t^{(t+2+2k)/2} 2^{-3(t-1)/4} 3^{-(t-1)/4}$ and we recall that we have defined $\beta = \pi/ \sqrt{6}$. 
	
	\begin{lem}
		\label{lem:uniform_bound}
		For any $\delta>0$, there exists a polynomial $R_k(y)$ and $\theta>0$, such that 
		\begin{equation}
			|g_k(y)-g_{n,k}(y)|<n^{\delta-1/4} R_k(y)e^{-\theta y}. 
		\end{equation}
		In particular the sequence of functions $g_{n,k}$ converges uniformly to $g_k$ as $n\to\infty$. 
	\end{lem}
	
	\begin{proof}
		As in the proof of Lemma~\ref{lem:gn-bound}, the analysis will depend on the value of $y$ relative to $n$. It will suffice to focus on nonnegative $y$. 
		
		\noindent
		\textbf{Case 1}: $n/2<\ell_n(y)$: \\
		Recall that $\ell_n(y)$ is defined in \eqref{def-ln}.
		By \eqref{eq:inequality_l_n(y)}, we obtain
		\begin{align*}
			g_k(y) < \frac{2}{n} \ell_n(y)\kappa y^ke^{-\beta ty} 
			< \frac{1}{\sqrt{n}} 4yt\kappa y^ke^{-\beta ty}.
		\end{align*}
		Using Lemma~\ref{lem:gn-bound}, we have $g_{n,k}(y) < Q_k(y)e^{-\theta y}$, and by using similar analysis we get 
		\begin{equation}
			g_{n,k}(y)< \frac{1}{\sqrt{n}} 4ytQ_k(y)e^{-\theta y}.
		\end{equation}
		The required result then follows by the triangle inequality.
		
		\noindent
		\textbf{Case 2}: $0 \le \ell_n(y) \le n/2$: \\
		We will apply the triangle inequality to successive differences of the following functions.
		\begin{align*}
			h_n^{(0)}(y)&=g_{n,k}(y),\\
			h_n^{(1)}(y)&= (yt)^k n^{(t-1)/4} \frac{d_t(n - \ell_n(y))}{p(n)},\\
			h_n^{(2)}(y) &=  \frac{\kappa y^k n^{(t+3)/4}}{(n-\ell_n(y))^{(t+3)/4}} \exp \left(-\pi\sqrt{\frac{2}{3}} \Big(\sqrt{n}-\sqrt{n - \ell_n(y)}\Big) \right),\\
			h_n^{(3)}(y) &= \kappa  y^k\exp \left(-\pi\sqrt{\frac{2}{3}} \Big(\sqrt{n}-\sqrt{n - \ell_n(y)}\Big) \right),\\
			h_n^{(4)}(y) &= g_k(y).
		\end{align*}
		From the proof of Lemma~\ref{lem:gn-bound}, it follows that $ n^{(t-1)/4} d_t(n - \ell_n(y))/p(n)$ is bounded by an exponential in $y$. Also $|(yt)^k-\ell_n(y)^k/n^{k/2}|$ is $O(n^{-1/2} \allowbreak Q_1(y))$ by expanding using binomial theorem, where $Q_1$ is some polynomial. This bounds $h_n^{(0)} - h_n^{(1)}$.
		
		To bound $h_n^{(2)}-h_n^{(1)}$, we will need the error terms in the asymptotics of $d_t(n)$ and $p(n)$.  From Meinardus' theorem~\cite[Theorem 6.2]{andrews-1976},
		\begin{align*}
			p(n) &=  \frac{\exp \left(\pi \sqrt{\frac{2n}{3}} \right)}{4n\sqrt{3}} \bigg(1+O(n^{-1/2}) \bigg), \\
			d_t(n) &= \frac{t^{(t+2)/2} \exp \left(\pi \sqrt{\frac{2n}{3}} \right)}{2^{(3t+5)/4} 3^{(t+1)/4}n^{(t+3)/4 } } \bigg(1+O(n^{\delta - 1/4}) \bigg),
		\end{align*}
		where $\delta$ is any positive real number. It is clear that
		\begin{equation*}
			\frac{1}{p(n)}= \frac{4n\sqrt{3}}{\exp \left(\pi \sqrt{\frac{2n}{3}} \right)}\left( 1+O(n^{-1/2}) \right).
		\end{equation*}
		Multiplying $d_t(n)$ by $1/p(n)$ above and comparing with $h_n^{(2)}$, we see that
		\begin{align*}
			|h_n^{(2)}(y)-h_n^{(1)}(y)| = O(n^{\delta-1/4}) h_2(y,n) = O(n^{\delta-1/4})Q_2(y)e^{-\theta y},
		\end{align*}
		where $Q_2$ and $\theta$ are chosen according to the proof of Lemma~\ref{lem:gn-bound}. 
		
		To obtain the required bound for $h_n^{(2)}-h_n^{(3)}$, we again use the binomial theorem for $(1 - \ell_n(y)/n)^{-(t+3)/4}$ (assuming $\ell_n(y)<n/2$).
		
		The bound for $h_n^{(3)}-h_n^{(4)}$ holds by noting that $|e^u-1|<2|u|$ for small enough $|u|$. Thus 
		\begin{multline*}
			\left| \exp \left(-\pi\sqrt{\frac{2}{3}} \left(\sqrt{n}-\sqrt{n - \ell_n(y)}\right) \right) -\exp(-\beta ty) \right|  \\
			< 2 e^{-\beta ty}\left|\pi\sqrt{\frac{2}{3}} \Big(\sqrt{n}-\sqrt{n - \ell_n(y)}\Big)- \beta ty \right| = O(n^{-1/2}) Q_3(y)e^{-\beta ty},
		\end{multline*}
		for large enough $n$, once again using the binomial theorem.
		This completes the proof.
	\end{proof}
	
	\begin{lem}
		\label{lem:conv_of_integal}
		For any natural number $k$, 
		\begin{equation}
			\lim_{n\to \infty} \int_{0}^{\infty} \;\text{d}y \; \psi_n(y)(g_k(y)-g_{n,k}(y)) =0.
		\end{equation}
	\end{lem}
	\begin{proof}
		For $t \geq 4$, \cite[Corollary 7]{anderson-2008} implies that 
		$\numcore(n) = O(n^{(t-3)/2+\epsilon})$ for any $\epsilon>0$.\footnote{Moreover, we can choose $\epsilon= 0$ for $t \geq 6$.}
		For $t=3$, we have from \cite[just before Lemma 3]{granville-ono-1996} that
		\[
		c_3(n) = \sum_{d | 3n+1} \left( \frac{d}{3} \right),
		\]
		where $\left( \frac{\cdot}{\cdot} \right)$ is the Legendre symbol. Therefore, $c_3(n) \leq d(3n+1)$, where $d(n)$ is the divisor function. A standard result in analytic number theory (see \cite[Section 13.11, Eq. (31)]{apostol-1976}, for example) says that $d(n) = o(n^{\epsilon})$ for any $\epsilon > 0$. Thus, for $t \geq 3$, we have 
		\[
		\psi_n(y)
		= \frac{\numcore(\ell_n(y))}{n^{(t-3)/4}}
		= O(n^{\epsilon}(1+y^r))
		\]
		for some integer $r$. The result then follows from Lemma~\ref{lem:uniform_bound} by choosing $\delta+\epsilon<1/4$, and noting that exponentially decaying functions have finite integral.
		
		Now let $t=2$. Then we have 
		\begin{equation}
			\psi_n (y)= \begin{cases}
				n^{1/4} & \ell_n(y) =\binom{i}{2} \text{ for $i\in \mathbb{N}$.} \\
				0 & \text{otherwise.}
			\end{cases}
		\end{equation}
		Therefore, $\psi_n(y)$ is supported at the intervals
		\begin{equation}
			\frac{\binom{i}{2}-\epsilon(n)}{t\sqrt{n}}\le y < \frac{\binom{i}{2}-\epsilon(n)}{t\sqrt{n}} + \frac{1}{\sqrt{n}}.
		\end{equation}
		From Lemma~\ref{lem:uniform_bound}, by choosing $0<\theta'<\theta$ and a large enough constant $B$, we may assume 
		\begin{equation}
			|g_k(y)-g_{n,k}(y)|<n^{\delta-\frac{1}{4}}Be^{-\theta' y}. 
		\end{equation}
		Thus the required integral is bounded by 
		\begin{equation}
			\begin{split}
				\int_{0}^{\infty} \;\text{d}y \; \psi_n(y)n^{\delta-1/4} Be^{-\theta'y} <&
				B \sum_{i=1}^{\infty} \frac{1}{\sqrt{n}}\left(n^{\delta} e^{-\theta' \frac{\binom{i}{2}-\epsilon(n)}{t\sqrt{n}}} \right)\\
				=& B n^{\delta-1/2} e^{\theta' \frac{\epsilon(n)}{t\sqrt{n}}}\sum_{i=1}^{\infty} \left(e^{-\frac{\theta'}{t\sqrt{n}} } \right)^{\binom{i}{2}},
			\end{split}
		\end{equation}
		where the first inequality comes from bounding the integral by an upper Riemann sum with intervals of length $n^{-1/2}$ supported precisely at the support of $\psi_n(y)$, and noting that $e^{-\theta'y}$ is a decreasing function.
		For us, a crude bound for the above sum will suffice, although $\sum_i q^{\binom{i}{2}} $ is a Jacobi theta function (in our case $q=e^{-\theta'/(t\sqrt{n}) }$) and therefore more sophisticated asymptotic results are available.
		Observe that there are at most $2n^{1/4}$ triangular numbers between $s\sqrt{n}$ and $(s+1)\sqrt{n}$ for any natural number $s$. Thus the above sum is bounded by
		\begin{align}
			B n^{\delta-1/2} e^{\frac{\epsilon(n)}{t\sqrt{n}}}\sum_{s=0}^{\infty} 2n^{1/4} \left( e^{-\frac{\theta'}{t\sqrt{n}} } \right)^{s\sqrt{n}} 
			=&  2 B n^{\delta-1/4} e^{\frac{\epsilon(n)}{t\sqrt{n}}} \frac{1}{1-e^{\theta'/t}}
		\end{align}
		which approaches zero as $n\to \infty$ for $\delta <1/4$.
	\end{proof}
	
	We will now prove convergence of the moments of $Y_n/\sqrt{n}$ to those of $Y$.
	
	\begin{lem}
		\label{lem:moments-conv}
		For all positive integer $k$, 
		\begin{equation}
			\lim_{n\to \infty}\mathbb{E} \left( \frac{Y_n}{\sqrt{n}} \right)^k = \left(\frac{\sqrt{6}}{\pi} \right)^k \prod_{i=1}^{k} \left(\frac{t-1}{2}+k-i \right).
		\end{equation}
	\end{lem}
	
	\begin{proof}
		From the definitions we get 
		\begin{align*}
			\mathbb{E}\left( \frac{Y_n}{\sqrt{n}} \right)^k =\sum_{i=0}^{\infty}\frac{i^k\mu_{n,t}(i)}{n^{k/2}},
		\end{align*}
		where $\mu_{n,t}$ is given in \eqref{def-mu}.
		Since $d_t(k)=0$ for all $k$ not divisible by $t$, we have
		\begin{align*}
			\mathbb{E}\left( \frac{Y_n}{\sqrt{n}} \right)^k &= \sum_{j=0}^{\infty} \frac{(jt+\epsilon(n))^k\mu_{n,t}(jt+\epsilon(n)) }{n^{k/2}}.
		\end{align*}
		Recall that
		$\psi_n$ and $g_{n,k}$ are defined in \eqref{def-psi-n} and \eqref{def-gnk} respectively. 
		Since $\psi_n\cdot g_{n,k}$ is a step function with step length $1/\sqrt{n}$, we are able to write this sum as the integral
		\begin{align*}
			\mathbb{E}\left( \frac{Y_n}{\sqrt{n}} \right)^k &=  \int_{0}^{\infty} \;\text{d}y \; \psi_n(y)g_{n,k}(y).
		\end{align*}
		It is useful to split the integral as
		\begin{equation*}
			\begin{split}
				\mathbb{E}\left( \frac{Y_n}{\sqrt{n}} \right)^k
				&= \int_{0}^{\infty} \;\text{d}y \; \psi_n(y)g_{k}(y) +\int_{0}^{\infty} \;\text{d}y \; \psi_n(y)(g_{n,k}(y)-g_n(k)),
			\end{split}
		\end{equation*}
		where $g_k$ is the explicit function in \eqref{def:g_k}.
		Using Lemma \ref{lem:conv_of_integal}, we know that the second term goes to zero as $n\to \infty$. Using integration by parts, the first term is
		\begin{align}
			\label{integ-by-parts}
			\int_{0}^{\infty} \;\text{d}y \; \psi_n(y)g_{k}(y) = - 
			\int_{0}^{\infty} \;\text{d}y \;g_k'(y) \int_{0}^{y} \;\text{d}u \; \psi_n(u).
		\end{align}
		Plugging in $\psi_n(u)=\numcore(\ell_n(u))/n^{(t-3)/4}$, we get
		\begin{align*}
			\int_{0}^{y} \;\text{d}u \; \psi_n(u) &= \sum_{i=0}^{\lfloor y\sqrt{n}\rfloor}\frac{\numcore(it+\epsilon(n))}{n^{(t-1)/4}}+ O\bigg(\frac{\numcore(\ell_n(y))}{n^{(t-1)/4}} \bigg).
		\end{align*}
		By Corollary~\ref{cor:numcore-size}, we know that $\numcore(n)=O (n^{(t-2)/2})$. Thus, using Proposition~\ref{prop:coresum-formula}, we get 
		\begin{align*}
			\int_{0}^{y} \;\text{d}u \; \psi_n(u) &=\frac{\coresum (\ell_n(y))}{n^{(t-1)/4}}+ O\bigg(\frac{(ty)^{(t-2)/2}}{n^{1/4}}\bigg)\\
			&=\frac{(2\pi)^{(t-1)/2}(yt)^{(t-1)/2} }{t^{(t+2)/2}\Gamma(\frac{t+1}{2})} + O\bigg(\frac{(ty)^{(t-2)/2}}{n^{1/4}} \bigg),
		\end{align*}
		where we have used Theorem \ref{thm:sum} for the last equality. 
		Note that $g_k$ is independent of $n$ and exponentially decaying, and the error term is uniformly bounded by a power of $y$.
		Thus, the error term, when integrated against $g'_k$, is $O(n^{-1/4})$. 
		Substituting this formula back in \eqref{integ-by-parts}, we obtain  
		\[ 
		\lim_{n\to\infty} \int_{0}^{\infty}\;\text{d}y \; \psi_n(y)g_{k}(y) =-\frac{(2\pi)^{(t-1)/2} }{t^{(t+2)/2}\Gamma(\frac{t+1}{2})}\int_{0}^{\infty} \;\text{d}y \; (yt)^{(t-1)/2} g_k'(y).
		\]
		The integrals are easily evaluated using the standard gamma integral, and after combining them, we obtain
		\begin{align*}
			\lim_{n\to \infty}\mathbb{E}\left( \frac{Y_n}{\sqrt{n}} \right)^k
			&= \frac{(2\pi)^{(t-1)/2} }{t^{(t+2)/2} \Gamma(\frac{t+1}{2})}\times \frac{ (t-1) \kappa}{2\beta^{k+(t-1)/2} t^{k}} \times \Gamma \left(k+\frac{t-1}{2} \right)\\
			&= \left(\frac{\sqrt{6}}{\pi} \right)^k \frac{\Gamma(\frac{t-1}{2}+k)}{\Gamma(\frac{t-1}{2})},
		\end{align*}
		where we have used that $\kappa=t^{(t+2+2k)/2} 2^{-3(t-1)/4} 3^{-(t-1)/4}$ from \eqref{def:g_k},
		proving the result.
	\end{proof}
	
	For the gamma distribution with shape parameter $\alpha$ and rate parameter $\beta$ whose density is given in \eqref{gamma-density}, the $k$'th moment $m_k$ can be calculated using the Gamma distribution to be
	\begin{equation}
		\label{gamma-moments}
		m_k = \frac{\Gamma(k + \alpha)}{\beta^k \Gamma(\alpha)}.
	\end{equation}
	For the relevant parameters of $\alpha$ and $\beta$ in our this, this is exactly the result in Lemma~\ref{lem:moments-conv}.

	\begin{proof}[Proof of Theorem~\ref{thm:conv}]
		From the formula for the moments of the gamma distribution in \eqref{gamma-moments}, one can verify that the moment generating function of the gamma distribution is
		\[
		\sum_{k=0}^\infty m_k \frac{t^k}{k!} = \left(1 - \frac{t}{\beta} \right)^{-\alpha}, \quad \text{for $t < \beta$},
		\]
		and therefore it is determined by its moments (see \cite[Theorem 30.1]{billingsley-1995}, for example). Using Lemma~\ref{lem:moments-conv}, we have shown that all the moments of $Y_n/\sqrt{n}$ exist and converge to those of $Y$, which is gamma distributed. Again appealing to a standard result on weak convergence~\cite[Theorem 30.2]{billingsley-1995}, we see that $Y_n/\sqrt{n}$ converges weakly to $Y$.
	\end{proof}

	\begin{proof}[Proof of Corollary~\ref{cor:expectation}]
		By Theorem~\ref{thm:conv}, $Y_n/\sqrt{n}$ converges in moments to $Y$. 
		We get the required constant $\mathbb{E}Y = \alpha/ \beta = (t-1)\sqrt{6}/(2\pi)$ by setting $k=1$ in \eqref{gamma-moments}. 
	\end{proof}

	\section{Hook lengths of random cells of random partitions}
	\label{sec:hooks}
	
	The main result of this section is that for large enough $n$, the modulo class of the hook length of a uniformly random cell of a uniformly random partition is approximately the uniform distribution on the modulo classes $\{0,1,\dots , t-1 \}$.

	\subsection{Consequences of Theorem~\ref{thm:conv}}
	In Corollary~\ref{cor:expectation} we have observed that the expected size of the $t$-core of a uniformly random partition of size $n$ is $O(\sqrt{n})$.  
	Therefore, Theorem~\ref{thm:partn-divn} suggests that most of the mass of a random partition lives in the corresponding $t$-divisible partition. 
	This idea allows us to obtain results about hook lengths of random partitions which might be difficult to obtain otherwise.
	
	\begin{prop}
		\label{prop:mod_zero_case}
		For a uniformly random cell $c$ of a uniformly random partition $\lambda$ of $n$, the probability that the hook length of $c$ in $\lambda$ is divisible by $t$ is $1/t+O(n^{-1/2})$.
	\end{prop}
	
	\begin{proof}
		It is known that removing a $t$-rim hook reduces the number of hook lengths divisible by $t$ by exactly one; see \cite[Theorem 3.3]{olsson}, for example.\footnote{It is also possible to see this directly using the $t$-runner abacus.} 
		So, for any partition $\lambda\vdash n$
		\begin{equation}
			\#\{c\in \lambda \mid t\vert h_c \} = \frac{n-|\core_t(\lambda)|}{t}.
		\end{equation}
		Hence, the number of cells divisible by $t$ is completely determined by the size of the partition and its $t$-core. For a uniformly random cell $c$ of a uniformly random partition $\lambda$ of $n$, we find that the probability that the hook length of $c$ in $\lambda$ is divisible by $t$ tends to
		\begin{align*}
			&\lim_{n\to \infty} \frac{n-\mathbb{E}(|\core_t(\lambda)|)}{tn}
			= \frac{1}{t}-\lim_{n \to \infty}\frac{t-1}{2t\pi}\sqrt{\frac{6}{n}}
		\end{align*}
		as $n\to \infty$, completing the proof.
	\end{proof}

	\begin{prop}
		\label{prop:mod-pmi}
		Fix $r\in \{1,\dots, t-1 \}$. Then, for a uniformly random cell $c$ of a uniformly random partition $\lambda$ of $n$, the probability that the hook length of $c$ in $\lambda$ is congruent to either $\pm r \mod t$ is $2/t + O(n^{-1/2})$ if $r \ne t/2$ and $1/t + O(n^{-1/2})$ if $r=t/2$.
	\end{prop}
	
	\begin{proof}
		By Proposition~\ref{prop:core-facts}(3), removing a $t$-rim hook from $\lambda$ corresponds to exchanging the positions of $w_i = 1$ and $w_{i+t} = 0$ in the corresponding 1-runner abacus $w$. The hook length of cells in $\lambda$ which do not involve either $i$ or $i+t$ are clearly unaffected. Moreover, the hook length modulo $t$ does not change for all cells $(i,j)$ if $j > i+t$ and $(k,i+t)$ if $k < i$. 
		
		Now consider the positions $w_j$ for $i < j < i+t$. If $w_j=0$, exchanging positions $i$ and $i+t$ removes a cell with hook length $j-i$ and if $w_j=1$ then a cell with hook length $t-(j-i)$.  In either case, if $2(j -i) \ne t$, the number of cells congruent to $\pm (j-i)$ modulo $t$ reduces by 1. Similarly, if $2(j-i)=t$, the number of cells congruent to $t/2$ modulo $t$ reduces by 1. We can now mimic the calculation in Proposition~\ref{prop:mod_zero_case} to prove the result.
	\end{proof}
	
	Using the results in Section~\ref{sec:distrib}, we are only able to prove the results in Propositions~\ref{prop:mod_zero_case} and \ref{prop:mod-pmi}. 
	To obtain the stronger result stated in Theorem~\ref{thm:random hook length}, we will need to use a natural action of the symmetric group $S_t$.
	
	\begin{example}
		Consider the partition $\lambda = (6,4,3,1)$ and let $t=4$. 
		\[
		\ytableausetup{centertableaux}
		\begin{ytableau}
			9 & 7 & 6 & 4 & 2 & 1 \\
			6 & 4 & 3 & 1 \\
			4 & 2 & 1\\
			1
		\end{ytableau}
		\]
		There are 5 (resp. 2) cells whose hook lengths are congruent to 1 (resp -1) modulo 4. Removing the 4-rim-hook corresponding to the cell (3,1) leads to the partition
		$\lambda' = (6,4)$.
		The following table illustrates how the hook lengths modulo 4 are affected.
		
		\begin{center}
			\begin{tabular}{|c|p{4cm}|p{4cm}|}
				\hline
				i & No. of hook lengths in $\lambda \equiv i \mod 4$ & No. of hook lengths in $\lambda' \equiv i \mod 4$ \\
				\hline
				0 & 3 & 2\\
				1 & 5 & 3 \\
				2 & 4 & 3 \\
				3 & 2 & 2 \\
				\hline
			\end{tabular}
		\end{center}
		\noindent
		Therefore, we have removed two cells congruent to 1 modulo 4, but none congruent to 3 modulo 4.
	\end{example}

	\subsection{A natural action of $S_t$}
	We define the action of $S_t$ on  $\mathcal{D}_t$, the set of $t$-divisible partitions, as follows. For any $\sigma\in S_t$ and $t$-divisible  partition $\nu$ with $t$-quotient $(\nu^0, \nu^1,\dots, \nu^{t-1})$, define $\sigma \nu $ be the $t$-divisible partition corresponding to the $t$-quotient $(\nu^{\sigma_0}, \nu^{\sigma_1},\dots ,\nu^{\sigma_{t-1}})$.
	Note that the above action preserves the size of the $t$-divisible partition by the last equality in Theorem \ref{thm:partn-divn}. 
	
	\begin{defn}
		The {\em $b$-smoothing} of a $t$-divisible partition $\nu$, denoted $C_\nu^b$, is the union of cells in the Young diagram of $\nu$ whose corresponding $(0,1)$ pairs are at least $(b+1)$ columns apart in the $t$-runner abacus of $\nu$.
	\end{defn}
	
	Note that the $(-1)$-smoothing of a a $t$-divisible partition $\nu$ is $\nu$ itself.
	
	\begin{example}
		\label{eg:action-of-S3}
		Consider the action of $S_3$ on the $3$-divisible partition $\nu =(7,3,2)$, whose 3-runner abacus is given by 
		\[
		\begin{matrix}
			\nu^0 = & (\dots, & 1, & 1, & 0, & \underline{0}, & 0, & 1, & 0, & \dots)\\
			\nu^1 = & (\dots, & 1, & 1, & 0, & 1, & 0, & 0, & 0, & \dots)\\
			\nu^2 = & (\dots, & 1, & 1, & 1, & 0, & 0, & 0, & 0, & \dots)
		\end{matrix}
		\]
		The orbit of $\nu$ under the action of $S_3$ and their $b$-smoothings for $0\le b \leq 2$ are described in the following table. 
		
		\def\arraystretch{1.2}
		\begin{center}
			\begin{tabular}{|c|p{2cm}|p{1cm}|p{1cm}|p{1cm}| }
				\hline
				$\sigma \in S_3$& $\sigma\nu\in \mathcal{D}_3$  & $C_{\sigma\nu}^0$ & $C_{\sigma\nu}^1$&$C_{\sigma\nu}^2$ \\
				\hline
				$123$ & $(7,3,2)$ & $(7,2)$ &$(4)$& $(2)$ \\
				$132$ & $(7,4,1) $ & $(7,2)$ & $(4)$&$(2)$ \\
				$213$ & $(8,2,2) $ & $(7,2)$ & $(4)$&$(2)$ \\
				$231$ & $(8,4) $ & $(7,2)$ & $(4)$&$(2)$ \\
				$312$ & $(9,2,1) $ & $(7,2)$ & $(4)$&$(2)$ \\
				$321$ & $(9,3) $ & $(7,2)$ & $(4)$&$(2)$ \\
				\hline
			\end{tabular}
		\end{center}
		\noindent
		We see that $C_{\sigma\nu}^b$ turns out to be a subpartition of $\sigma\nu$ for each $\sigma\in S_3$. Moreover, this example suggests that $C_{\sigma\nu}^b$ does not depend on $\sigma$. We prove these results in Proposition~\ref{prop:smoothing-subpartition} below.
	\end{example}
	
	\begin{rem}
		There are several interpretations of the $0$-smoothing of the partition. In Example~\ref{eg:action-of-S3}, note that the intersection of the cells in $\sigma \nu$ over all $\sigma \in S_3$ is precisely the Young diagram corresponding to the partition $C_{\sigma\nu}^0 = (7,2)$.
		Moreover, $(7,2)$ is denoted by the 3-runner abacus given below. 
		\[
		\begin{matrix}
			\bar{\nu}^0 = & (\dots, & 1, & 1, & 1, & \underline{1}, & 0, & 1, & 0, & \dots)\\
			\bar{\nu}^1 = & (\dots, & 1, & 1, & 0, & 0, & 0, & 0, & 0, & \dots)\\
			\bar{\nu}^2 = & (\dots, & 1, & 1, & 0, & 0, & 0, & 0, & 0, &\dots)
		\end{matrix}
		\]
		By comparing with the $3$-runner abacus of $\nu$, we see that $C_{\sigma\nu}^0$ has been obtained from $\nu$ by justifying all the $1$'s in $\nu$ upwards. Both these observations turn out to hold in general.
	\end{rem}
	
	\begin{prop}
		\label{prop:smoothing-subpartition}
		The set of cells $C_\nu^b$ is a connected subpartition of $\nu$. Moreover $C_\nu^b=C_{\sigma\nu}^b$ for all $\sigma\in S_t$ and $b \geq 0$.
	\end{prop}
	
	\begin{proof}
		For any cell $c$ in the $t$-divisible partition $\nu$, we define the cell one place above and one place left by $c_U$ and $c_L$ respectively. It is enough to show that for any cell $c\in C_\nu^b$, $c_U, c_L \in C_\nu^b$. 
		By Proposition~\ref{prop:core-facts}(3), consider the pair $(\nu^a_i, \nu^b_j) = (0,1)$ pair (where $j > i+b$) corresponding to $c$ in the $t$-runner abacus of $\nu$.  Now, observe that $c_L$ is represented by the pair $(\nu^{a'}_{i'}, \nu^b_j) =(0,1)$ with either $i' < i$, or  $i' = i, a' < a$. Thus $c_L\in C_\nu^b$. A  similar argument ensures that $c_U \in C_\nu^b$.
		
		The proof of the second statement follows from the fact that the definition of the $b$-smoothing depends only on relative column positions and $\sigma$ only interchanges the rows.
	\end{proof}
	
	We illustrate Proposition~\ref{prop:smoothing-subpartition} by an example.
	
	\begin{example}
		Continuing Example~\ref{eg:action-of-S3}, let $c$ be the cell corresponding to the boxed $(0,1)$ pair
		\[
		\begin{matrix}
			\nu^0 = & (\dots, & 1, & 1, & 0, & \underline{0}, & 0, & 1, & 0, & \dots)\\
			\nu^1 = & (\dots, & 1, & 1, & \boxed{0}, & \boxed{1}, & 0, & 0, & 0, & \dots)\\
			\nu^2 = & (\dots, & 1, & 1, & 1, & 0, & 0, & 0, & 0, & \dots)
		\end{matrix}.
		\]
		Then $c_L$ and $c_U$ are given by the $(0,1)$ pairs with hats and tildes respectively,
		\[
		\begin{matrix}
			\nu^0 = & (\dots, & 1, & 1, & \hat{0}, & \underline{0}, & 0, & \tilde{1}, & 0, & \dots)\\
			\nu^1 = & (\dots, & 1, & 1, & \tilde{0}, & \hat{1}, & 0, & 0, & 0, & \dots)\\
			\nu^2 = & (\dots, & 1, & 1, & 1, & 0, & 0, & 0, & 0, & \dots)
		\end{matrix}.
		\]
	\end{example}
	
	For any cell $c\in C_\nu^b$, define $h_c^\nu$ to be the hook length of the cell $c$ in the partition $\nu$.
	
	\begin{lem} 
		\label{lem:hooklength_t-div}
		Let $\nu$ be a uniformly random $t$-divisible partition of $n$ and $b \geq 0$.
		For a uniformly random cell $c\in C_\nu^b$, the probability that the hook length $h_c^\nu$  is congruent to $i$ modulo $t$, where $i \neq 0$, is independent of $i$.
	\end{lem}
	
	\begin{proof}
		Fix  $\nu\in \mathcal{D}_t$ and $b\in \mathbb{N}$.
		For any $\sigma\in S_t$, there is a natural bijection from the cells of $C_\nu^b$ to those in $C_{\sigma\nu}^b$, defined by mapping the appropriate $(0,1)$ pairs $(\nu^\ell_j, \nu^m_k) \mapsto (\nu^{\sigma(\ell)}_j, \nu^{\sigma(m)}_k)$.
		Note that the hook length of the cell corresponding to the pair $(\nu^\ell_j, \nu^m_k)$ is congruent to $(\ell-m)$ mod $t$. Thus if $\ell \neq m$, the group action of $S_t$ takes the hook length of this pair to all possible nonzero hook lengths modulo $t$ an equal number of times.
		
		Now consider the set of all $t$-divisible partitions of $n$. Under the action of $S_t$, this set will be partitioned into disjoint orbits. By the argument above, the number of cells with hook lengths congruent to any nonzero modulo class of $t$ is the same within any orbit. Therefore, we obtain our desired result.
	\end{proof}

	\begin{defn}
		\label{def:action of S_t}
		We define the action of $\sigma \in S_t$ on $\mathcal{P}$ by 
		$\sigma \lambda = \Delta_t^{-1}(\rho, \allowbreak \sigma \nu)$,
		where $\Delta_t (\lambda) =(\rho, \nu)$ is defined in Theorem~\ref{thm:partn-divn}. 
	\end{defn}
	
	\begin{rem}
		\label{rem:Shift_description}
		We give an alternate description of $\sigma \lambda$ from Definition~\ref{def:action of S_t} using Remark~\ref{rem:t-runner_shift}.
		Let the core $\rho$ be determined by its positions of justification
		$(p_0,\dots, p_{t-1})$ as explained in Proposition~\ref{prop:t-runner}(1).
		Then $\sigma \lambda$ is the partition corresponding to 
		the $t$-runner
		\[
		(\delta ^{p_{\sigma(0)}-p_0} \lambda^{\sigma(0)},\delta ^{p_{\sigma(1)}-p_1} \lambda^{\sigma(1)}, \dots ,\delta ^{p_{\sigma(t-1)}-p_{t-1}} \lambda^{\sigma(t-1)} ),
		\]
		where $(\lambda^0, \dots, \lambda^{t-1})$ is the $t$-runner abacus for $\lambda$.
	\end{rem}

	Let $\lambda \in \mathcal{P}$ with $\Delta_t (\lambda) =(\rho, \nu)$, and $\rho$ be determined by  the $t$-tuple $(p_0,\dots, p_{t-1})$ and let $b_\lambda = \max_{1 \leq i < j \leq t-1} |p_i-p_j|$.
	We will denote $C_\nu^{b_\lambda}$ by $C_\lambda$ for brevity and call it the {\em canonical smoothing} of the $t$-quotient of $\lambda$. We illustrate $C_\lambda$ with an example.
	
	\begin{example}
		\label{eg:lambda-running}
		Let $\lambda=(10,3)$. Then $\Delta_3 (\lambda) =(\rho, \nu)$, where $\nu= (7,3,2)$ and $\rho=(1)$. The 3-runner abacus of $\rho$ is given by
		\[
		\begin{matrix}
			\rho^0 = & (\dots, & 1, & 1, & 1, & \underline{1}, & 0, & 0, & 0, & \dots)\\
			\rho^1 = & (\dots, & 1, & 1, & 1, & 0, & 0, & 0, & 0, & \dots)\\
			\rho^2 = & (\dots, & 1, & 1, & 0, & 0, & 0, & 0, & 0, & \dots)
		\end{matrix},
		\] 
		which implies $p_0=1,p_1=0, p_2=-1$ and $b_\lambda=2$. Therefore $C_\lambda=C_\nu^2=(2)$ from Example \ref{eg:action-of-S3}.
		The 3-runner abacus of $\nu$, with the cells $c_1, c_2$ of the canonical smoothing $C_\lambda \subset \nu$ marked by hats and tildes respectively, is given by
		\[
		\begin{matrix}
			\nu^0 = & (\dots, & 1, & 1, & \hat 0, & \underline{0}, & 0, & \tilde{\hat 1}, & 0, & \dots)\\
			\nu^1 = & (\dots, & 1, & 1, & \tilde 0, & 1, & 0, & 0, & 0, & \dots)\\
			\nu^2 = & (\dots, & 1, & 1, & 1, & 0, & 0, & 0, & 0, & \dots)
		\end{matrix}.
		\] 
		These cells are shown below in the Young diagram of $\nu$:
		\[
		\ytableausetup{centertableaux}
		\begin{ytableau}
			c_1 & c_2 &  & & & &\\
			& & \\
			& 
		\end{ytableau}.
		\]
		From Remark~\ref{rem:t-runner_shift}, the $t$-runner abacus for $\lambda$ is given by $\lambda^i=\delta^{-p_i}\nu^i$. 
		Therefore, the 3-runner of $\lambda=(10,3)$ is given by 
		\[
		\begin{matrix}
			\lambda^0 = & (\dots, & 1, & 1, & 1, & \underline{0}, & 0, & 0, & 1, & \dots)\\
			\lambda^1 = & (\dots, & 1, & 1, & 0, & 1, & 0, & 0, & 0, & \dots)\\
			\lambda^2 = & (\dots, & 1, & 1, & 0, & 0, & 0, & 0, & 0, & \dots)
		\end{matrix},
		\] 
		which can be easily verified by direct computation.
	\end{example}

	\begin{prop}
		\label{prop:injection_of_cells}
		Let $\lambda$ be a partition with $\Delta_t (\lambda) =(\rho, \nu)$.
		Then there exists an injective map $\phi$ that takes the cells in $C_\lambda$ to the cells in $\lambda$ such that for any cell $c\in C_\lambda$, the hook lengths
		$h^\lambda_{\phi(c)}\equiv h^\nu_c \mod t$.
	\end{prop}
	
	\begin{proof}
		Let $c \in C_\lambda\subset \nu$ be given by the pair $(\nu^a_i, \nu^b_j) =  (0,1)$ where $j-i > b_\lambda$ by definition. Let $\rho$ be determined by the justification positions $(p_0,\dots ,p_{t-1})$, where $\sum_{i=0}^{t-1}p_i=0$ by Proposition \ref{prop:t-runner}(1). Then define 
		\[
		\phi(c) := (\nu^a_{i+p_a}, \nu^b_{j+p_b}).
		\]
		By Remark~\ref{rem:t-runner_shift}, $\nu^a_{i+p_a} = 0$ and $\nu^b_{j+p_b} = 1$ so that this pair corresponds to a cell in $\lambda$. The map $\phi$ is clearly well-defined since $j+p_b>i+p_a$ and is moreover injective. Since the modulo class of the hook length of a cell is purely determined by the rows of the corresponding $(0,1)$ pair, the result follows.
	\end{proof}
	
	\begin{example}
		Consider $\phi$ defined in the proof of Proposition~\ref{prop:injection_of_cells}.
		Continuing our running example of $\lambda = (10,3)$ in Example~\ref{eg:lambda-running}, the cells $f_1=\phi(c_1)$ and $f_2=\phi(c_2)$ are shown in the Young diagram of $\lambda$ below. 
		\[
		\ytableausetup{centertableaux}
		\begin{ytableau}
			f_2&  &  f_1& & & & & & &\\
			& & 
		\end{ytableau}
		\]
		Finally, note that $h^\nu_{c_1}=h^\lambda_{f_1}=9$ and $h_{f_2}^\lambda=11 \equiv 8=h_{c_2}^\nu \mod{3}$.
	\end{example}

	\subsection{Proof of the main theorem}
	
	The main idea is to show that the image of the canonical smoothing $\phi(C_\lambda)$ comprise of most of the cells in a uniformly random partition $\lambda$ of size $n$. 
	We will first show that $|\nu| -|C_\lambda|$ has very few number of cells in expectation.
	
	\begin{prop}\label{prop:bound_b_lambda}
		For any partition $\lambda$ with $\Delta_t (\lambda) =(\rho, \nu)$, 
		$|b_\lambda| \leq 2\sqrt{|\rho|}$.
	\end{prop}
	
	\begin{proof}
		The inequality clearly holds if $\rho = \emptyset$.
		Let $i,j \in \{0,\dots,t-1\}$ such that $b_\lambda = p_j - p_i$. Consider only the $(0,1)$ pairs in rows $i$ and $j$ of the $t$-runner abacus of $\rho$. Clearly, the number of such pairs is at least $\binom{b_\lambda}{2}$, and this is a lower bound for $|\rho|$. Now, since $b_\lambda \geq 2$ for nonempty $\rho$, the desired inequality follows.
	\end{proof}
	
	We will now estimate the number of cells of the region with small hook lengths. 
	
	\begin{lem}
		\label{lem:small-hooks}
		For $\lambda\vdash n$ and an integer $m$, the cardinality of the set $\mathcal{B}=\{c\in \lambda \mid h_c< m \}$ is less than $m\sqrt{2n}$.
	\end{lem}
	
	\begin{proof}
		It is enough to construct an injection from the set of unordered
		pairs of distinct cells in $\mathcal{B}$ to $\lambda \times [m] \times [m]$. We describe such a construction below.
		
		Suppose $c_1$ and $c_2$ are two cells in $\mathcal{B}$, and say that $c_1$ is to the west of $c_2$ and if both are in the same column let $c_1$ to be south of $c_2$. Let $a_2$ be the arm length of $c_2$ and $l_1$ be the leg length of $c_1$. Then map the pair $\{c_1,c_2\}$ to $(c,a_2,l_1)$, where $c$ is the unique cell in the intersection of the column containing $c_1$ and the row containing $c_2$. 
		
		This map is clearly injective since the cell $c$ describes the column (or row) in which the cell $c_1$ (resp. $c_2$) lives and $l_1$ (resp $a_2$) gives the exact location of $c_1$ (resp $c_2$) in the partition.
	\end{proof}
	
	\begin{prop}
		\label{prop:size nu-cnu}
		For any partition $\lambda \vdash n$ with $\Delta_t (\lambda) =(\rho, \nu)$, 
		\begin{equation}
			|\nu|-|C_\lambda|< \# \{c\in \nu \mid h_c<t(b_\lambda+1) \} = O( t (b_\lambda+1)\sqrt{n}).
		\end{equation}
	\end{prop}
	
	\begin{proof}
		From Proposition~\ref{prop:core-facts}(2), all cells with hook length at least $t(b_\lambda + 1)$ must be in $C_\lambda$. Since $C_\lambda$ is a connected subpartition of $\nu$ by Proposition~\ref{prop:smoothing-subpartition}, the first inequality follows. By Lemma~\ref{lem:small-hooks}, the second inequality immediately follows.
	\end{proof}
	
	We are now in a position to prove the main result of this section.
	
	\begin{proof}[{Proof of Theorem~\ref{thm:random hook length}}]
		Let $x_{i}(n)$ be the probability that a uniformly random cell $c$ of a uniformly random partition of $n$ has hook length congruent to $i$ modulo $t$.
		Using Proposition~\ref{prop:injection_of_cells} and Lemma~\ref{lem:hooklength_t-div}, we obtain by conditioning for any nonzero classes $i$ and $j$ modulo $t$, the difference between $x_i(n)$ and $x_j(n)$ in absolute value is upper bounded by the probability that $c \notin \phi(C_\lambda)$. Therefore, by Theorem~\ref{thm:partn-divn}, we have that
		\[
		|x_i(n)-x_j(n)|  \leq  \sum_{\substack{\lambda\vdash n \\ \Delta_t (\lambda) =(\rho, \nu)}} \frac{|\nu|-|C_\lambda|+|\rho|}{np(n)}.
		\]
		Now, by Proposition \ref{prop:bound_b_lambda}, Proposition~\ref{prop:size nu-cnu} and Corollary~\ref{cor:expectation},
		\[
		|x_i(n)-x_j(n)| < O(n^{-1/2}\mathbb{E}(\sqrt{|\core_t(\lambda)|})) + O(n^{-1/2}).
		\]
		Using the standard fact that $\mathbb{E}\sqrt{X}\le \sqrt{\mathbb{E} X }$ for a nonnegative random variable $X$, we see that the right hand side is $O(n^{-1/4})$.
		Since $x_0(n)+x_1(n)+\dots+x_{t-1}(n)=1$ and $x_0(n)=1/t+O(n^{-1/4})$, we obtain the required result.
	\end{proof}

\section*{Acknowledgements}
We thank D. Grinberg for suggesting the idea of the proof of Lemma \\
\ref{lem:small-hooks}.
This research was driven by computer exploration using the open-source mathematical software \texttt{Sage}~\cite{sagemath}.
The authors were partially supported by the UGC Centre for Advanced Studies. AA was also partly supported by Department of Science and Technology grant EMR/2016/006624.

\bibliography{cores}
\bibliographystyle{alpha}

\end{document}